\documentclass{amsart}[10pt]

\usepackage{amsrefs}
\usepackage[all]{xy}
\usepackage{syntonly}
\usepackage{hyperref}
\usepackage{amsfonts}
\usepackage{amssymb}
\usepackage{amsmath}
\usepackage{amsthm}
\usepackage[inline]{enumitem}
\usepackage{tikz}
\usepackage{graphics}
\usepackage{graphicx}
\usepackage{color}
\usepackage{comment}
\usepackage{caption,lipsum}
\usepackage{framed}

\usepackage{lineno}

\usepackage{multirow}

\usepackage{footmisc}

\newtheorem{theorem}{Theorem}[section]
\newtheorem{corollary}[theorem]{Corollary}

\newtheorem{lemma}[theorem]{Lemma}

\newtheorem{question}[theorem]{Question}

\newtheorem{fact}[theorem]{Fact}

\newtheorem{remark}[theorem]{Remark}

\specialcomment{com}{ \color{red} \textbf{COMMENT:} }{  \color{black}} 
\specialcomment{oldcom}{ \color{brown} }{\color{black}} 

\excludecomment{oldcom} 

\usepackage{stmaryrd}

\usepackage[margin=1in]{geometry}

\begin{document}
\title[]{Salce's problem on cotorsion pairs is undecidable}

\author{Sean Cox}
\email{scox9@vcu.edu}
\address{
Department of Mathematics and Applied Mathematics \\
Virginia Commonwealth University \\
1015 Floyd Avenue \\
Richmond, Virginia 23284, USA 
}



\thanks{Many thanks to Jan Trlifaj for helpful correspondence regarding cotorsion pairs and deconstructibility.}

\begin{abstract}
Salce~\cite{MR565595} introduced the notion of a \emph{cotorsion pair} of classes of abelian groups, and asked whether every such pair is \emph{complete} (i.e., has enough injectives and projectives); we refer to this as \textbf{Salce's Problem} (for \textbf{Ab}).  We prove that it is  consistent, relative to the consistency of Vop\v{e}nka's Principle (VP), that the answer is affirmative.  Combined with a previous result of Eklof-Shelah~\cite{MR2031314}, this shows that Salce's Problem for \textbf{Ab}, and in fact for \textbf{R-Mod} when $R$ is hereditary, is independent of the ZFC axioms (modulo the consistency of VP).  
\end{abstract}

\maketitle


\section{Introduction}\label{sec_Intro}

\emph{Cotorsion pairs}, also called \emph{cotorsion theories}, were introduced by Salce~\cite{MR565595} in the setting of abelian groups.  For a class $\mathcal{C}$ of $R$-modules,
\[
{}^\perp \mathcal{C}=_{\text{def}} \{ X \ : \ \text{Ext}^1_R(X,C) = 0 \text{ for all } C \in \mathcal{C}  \}
\]
and
\[
\mathcal{C}^\perp=_{\text{def}} \{ Y \ : \ \text{Ext}^1_R(C,Y) = 0 \text{ for all } C \in \mathcal{C} \}.
\]
A \textbf{cotorsion pair} is a pair $(\mathcal{A},\mathcal{B})$ of classes such that $\mathcal{A} = {}^\perp \mathcal{B}$ and $\mathcal{A}^\perp =\mathcal{B}$.  The cotorsion pair $(\mathcal{A},\mathcal{B})$ is called \textbf{complete} if, for every module $M$, there exists a short exact sequence
\[
\xymatrix{
0 \ar[r] & B \ar[r] & A \ar[r] & M \ar[r] & 0
}
\]
for some $A \in \mathcal{A}$ and $B \in \mathcal{B}$.\footnote{We also say the cotorsion pair \emph{has enough projectives} if this property holds.  Salce proved this is equivalent to the pair \emph{having enough injectives}; i.e., that for every module $M$, there is a short exact sequence $0 \to M \to B \to A \to 0$ with $A \in \mathcal{A}$ and $B \in \mathcal{B}$. }  Cotorsion pairs that are complete provide for a nice approximation theory.  The pair 
\begin{equation*}
\big( \text{Projective Modules}, \text{All Modules} \big)
\end{equation*}
 is the simplest example of a complete cotorsion pair.

Salce's original paper dealt exclusively with the category \textbf{Ab} of abelian groups (i.e., $\mathbb{Z}$-modules), and he asked whether every cotorsion pair in this category is complete (Problem 2 of \cite{MR565595}).  We shall refer to this as \textbf{Salce's Problem}, or sometimes as \textbf{Salce's Problem for Ab}, to distinguish it from the still-open generalized version (see Question \ref{q_GeneralSalceProblem}).  Our main result is:
\begin{theorem}\label{thm_MainVP_result}
It is consistent, relative to the consistency of Vop\v{e}nka's Principle (VP), that the answer to Salce's Problem is affirmative.  

In fact, it is consistent relative to VP that for any ring $R$ and any cotorsion pair $(\mathcal{A},\mathcal{B})$ of $R$-modules, if ${}^\perp B$ is downward closed under elementary submodules for all $B \in \mathcal{B}$,\footnote{This always holds if the ring is hereditary; indeed, if $R$ is hereditary and $X$ is any $R$-module, then ${}^\perp X$ is downward closed under all submodules.} then $(\mathcal{A},\mathcal{B})$ is both cogenerated and generated by a set, and hence\footnote{By Eklof-Trlifaj~\cite{MR1798574}, see below.} complete.  
\end{theorem}

\noindent VP is a well-studied principle that is equivalent to the assertion (scheme) that ``every subfunctor of an accessible functor is accessible" (see \cite{MR1294136} for many other characterizations).  The history of VP is amusing; the historical remarks in Chapter 6 of Ad\'{a}mek-Rosick\'{y}~\cite{MR1294136} refer to VP as ``a practical joke which misfired".  VP fits in the large cardinal hierarchy just below the huge cardinals.

Eklof-Shelah~\cite{MR2031314} proved that it is consistent, relative to the consistency of the Zermelo-Fraenkel axioms of mathematics (ZFC), that the answer to Salce's Problem is negative.  Together with Theorem \ref{thm_MainVP_result} this yields:
\begin{corollary}
Salce's Problem (for \textbf{Ab}, and more generally for \textbf{R-Mod} with $R$ hereditary) is independent of the ZFC axioms (assuming that ZFC + VP is consistent).
\end{corollary}

There are three main ingredients to the proof of Theorem \ref{thm_MainVP_result}.  First, if VP is consistent, then VP is also consistent with the statement 
\begin{equation}\label{eq_DiamondEverywhereIntro} 
\Diamond_\lambda(S) \text{ holds for all regular uncountable } \lambda \text{ and all stationary } S \subseteq \lambda,
\tag{*}
\end{equation} 
where $\Diamond_\lambda(S)$ is the well-known combinatorial principle isolated by Ronald Jensen (see the Appendix).  Because of a nice ``VP-preservation" theorem of Brooke-Taylor~\cite{MR2805294}, this part of the argument is a standard forcing construction, and is relegated to the Appendix.

The second key ingredient of Theorem \ref{thm_MainVP_result} is the following theorem, which employs VP and (set-theoretic) ``elementary submodel" arguments:

\begin{theorem}\label{thm_VP_Main}
Vop\v{e}nka's Principle implies that for every ring $R$ and every class $\mathcal{B}$ of $R$-modules:  if ${}^\perp B$ is downward closed under elementary submodules for all $B \in \mathcal{B}$, then there is a \textbf{set} $\mathcal{B}_0 \subseteq \mathcal{B}$ such that
\[
{}^\perp \mathcal{B} = {}^\perp \mathcal{B}_0.
\]
\end{theorem}

Rephrasing Theorem \ref{thm_VP_Main} in terms of cotorsion pairs yields:
\begin{corollary}\label{cor_VP_cogen}
Vop\v{e}nka's Principle implies that if $(\mathcal{A},\mathcal{B})$ is a cotorsion pair, and ${}^\perp B$ is downward closed under elementary submodules for all $B \in \mathcal{B}$, then $(\mathcal{A},\mathcal{B})$ is cogenerated by a set.\footnote{Using the terminology of G\"obel-Trlifaj~\cite{MR2985554}; ``generated by a set" in the terminology of Eklof-Trlifaj~\cite{MR1798574} and Salce~\cite{MR565595}.  See Section \ref{sec_Prelims} for details.}
\end{corollary}

\noindent Theorem \ref{thm_VP_Main} and Corollary \ref{cor_VP_cogen} can be generalized to other categories; see Section \ref{sec_GenOpen}.

The third ingredient of Theorem \ref{thm_MainVP_result} is:
\begin{theorem}[\v{S}aroch-Trlifaj~\cite{MR2336972}; Eklof-Trlifaj~\cite{MR1778163}]\label{thm_SarochTrlifaj}
If \eqref{eq_DiamondEverywhereIntro} holds, then every cotorsion pair that is cogenerated by a set, and whose left coordinate is downward closed under elementary submodules, is also generated by a set, and hence (by Eklof-Trlifaj~\cite{MR1798574}) complete.  
\end{theorem}

We also prove a result using ``only" supercompact cardinals, which are weaker in consistency strength than Vop\v{e}nka's Principle.  A class $\mathcal{C}$ of modules is a \textbf{Kaplansky class} (\cite{MR1926201}) if there exists a cardinal $\kappa$ such that, for every $C \in \mathcal{C}$ and every $X \subseteq C$ with $|X|<\kappa$, there is a $<\kappa$-presented submodule $C_0$ of $C$ such that $X \subseteq C_0$, and both $C_0$ and $C/C_0$ are in $\mathcal{C}$.

\begin{theorem}\label{thm_SC_Kaplansky}
If there is a proper class of supercompact cardinals, then for every set $\mathcal{S}$ of modules over any ring, ${}^\perp \mathcal{S}$ is a Kaplansky class.  
\end{theorem}

\noindent The key to Theorem \ref{thm_SC_Kaplansky} is Lemma \ref{lem_IntersectExt} (page \pageref{lem_IntersectExt}), which may be of interest elsewhere; roughly speaking, supercompactness of $\kappa$ yields many elementary submodels $\mathfrak{N}$ of the universe of size $<\kappa$ such that ``intersection with $\mathfrak{N}$" commutes with the Ext functor (for objects that $\mathfrak{N}$ has access to).

Kaplansky classes that are also closed under direct limits are in fact \emph{deconstructible} (see \cite{MR3010854}). And cotorsion pairs whose left coordinate is deconstructible are complete (see \cite{MR2985554}).  Hence:

\begin{corollary}\label{cor_SC}
If there is a proper class of supercompact cardinals, then all cotorsion pairs that are cogenerated by a set, and whose left coordinate is closed under direct limits, are complete.
\end{corollary}

Corollary \ref{cor_SC} can be viewed as a variant of a theorem of El Bashir~\cite{MR2252654}, who proved that under Vop\v{e}nka's Principle, every class of modules that is closed under direct sums \emph{and direct limits} is a covering class.  In particular, his theorem yielded that under Vop\v{e}nka's Principle, cotorsion pairs whose left coordinate is closed under direct limits are complete.  Corollary \ref{cor_SC} uses a weaker assumption than Vop\v{e}nka's Principle, but at the expense of only being able to prove it for cotorsion pairs cogenerated by a set.

Section \ref{sec_Prelims} includes preliminaries, Section \ref{sec_SC} proves Theorem \ref{thm_SC_Kaplansky}, and Section \ref{sec_VP} proves Theorem \ref{thm_MainVP_result}.  Sections \ref{sec_SC} and \ref{sec_VP} are entirely independent of each other, but Section \ref{sec_SC} appears first because it is slightly more concrete than the proof from Vop\v{e}nka's Principle.  Section \ref{sec_GenOpen} discusses generalizations and open problems, and Appendix \ref{app_VPDIAMOND} proves that CON(VP) implies CON(VP + \eqref{eq_DiamondEverywhereIntro}).

\section{Preliminaries}\label{sec_Prelims}

All notation and terminology agrees with Jech~\cite{MR1940513} and G\"obel-Trlifaj~\cite{MR2985554}.  By ``$R$-module" we will officially mean left $R$-module.

If $\mathcal{C}$ is a class, then both
\begin{equation}\label{eq_Cogen}
\Big( {}^\perp \mathcal{C}, \big( {}^\perp \mathcal{C} \big)^\perp  \Big)
\end{equation}
and
\begin{equation}\label{eq_Gen}
\Big( {}^\perp \big( \mathcal{C}^\perp \big), \mathcal{C}^\perp   \Big)
\end{equation}
are cotorsion pairs; \eqref{eq_Cogen} is called the cotorsion pair \textbf{cogenerated by $\boldsymbol{\mathcal{C}}$}, and \eqref{eq_Gen} is called the cotorsion pair \textbf{generated by $\boldsymbol{\mathcal{C}}$}, using the terminology of G\"obel-Trlifaj~\cite{MR2985554}.\footnote{The terminology is inconsistent across the literature; e.g., the meanings of ``cogenerated by" and ``generated by" given above are switched, for example, in Eklof-Trlifaj~\cite{MR1798574}.}  We say that a cotorsion pair is \textbf{cogenerated by a set} if it is of the form \eqref{eq_Cogen} where $\mathcal{C}$ is a set (as opposed to a proper class), and \textbf{generated by a set} if it is of the form \eqref{eq_Gen} where $\mathcal{C}$ is a set.  Around the turn of the millennium, Eklof and Trlifaj proved the following landmark theorem, which was key to one of the solutions of the Flat Cover Conjecture (\cite{MR1832549}):
\begin{theorem}[Eklof-Trlifaj~\cite{MR1798574}]\label{thm_ET_GenSet}
All cotorsion pairs generated by sets are complete.
\end{theorem}

 For structures $\mathfrak{A}$ and $\mathfrak{B}$ in a fixed first order signature, $\mathfrak{A} \prec \mathfrak{B}$ means that $\mathfrak{A}$ is an elementary substructure of $\mathfrak{B}$.  For a ring $R$, the language for $R$-modules consists of the usual language for abelian groups, together with, for each $r \in R$, a function symbol for scalar multiplication by $r$ (so for infinite $R$, the language for $R$-modules has cardinality $|R|$).  For an infinite cardinal $\lambda$, $H_\lambda$ refers to the set of all sets of hereditary cardinality $<\lambda$, and $\mathfrak{H}_\lambda$ refers to the structure $(H_\lambda,\in)$.  The reader more familiar with the $V_\alpha$ hierarchy can just as well use those instead. The equality $V_\lambda= H_\lambda$ holds for the closed unbounded class of $\lambda$ that are fixed points of the $\beth$ function.  The structure $\mathfrak{H}_\lambda$ is a $\Sigma_1$-elementary substructure of the universe for every uncountable cardinal $\lambda$.

For a regular uncountable cardinal $\kappa$ and a cardinal $\lambda \ge \kappa$, let $\wp^*_\kappa(H_\lambda)$ denote the set of $N \subset H_\lambda$ such that $|N|<\kappa$, $N \cap \kappa \in \kappa$, and $\mathfrak{N}=(N,\in)$ is an elementary substructure of $\mathfrak{H}_\lambda$.  Any such $\mathfrak{N}$ is extensional, so Mostowski's collapsing theorem applies.  For such $\mathfrak{N}$, let $\mathfrak{H}(\mathfrak{N})$ denote the Mostowski collapse of $\mathfrak{N}$, and let 
\[
\sigma_{\mathfrak{N}}: \mathfrak{H}(\mathfrak{N}) \to_{\text{iso}} \mathfrak{N} \ \prec \ \mathfrak{H}_\lambda 
\]
denote the inverse of the collapsing map, which can be viewed as an elementary embedding from $\mathfrak{H}(\mathfrak{N}) \to \mathfrak{H}_\lambda$.  For $b \in \mathfrak{N}$, $b_{\mathfrak{N}}$ will denote $\sigma^{-1}_{\mathfrak{N}}(b)$; i.e., $b_{\mathfrak{N}} \in \mathfrak{H}(\mathfrak{N})$ is the image of $b$ under the transitive collapsing map for $\mathfrak{N}$.  A set $S \subseteq \wp_\kappa^*(H_\lambda)$ is called \textbf{stationary (in $\boldsymbol{\wp^*_\kappa(H_\lambda)}$)} if, for every $p_1,\dots,p_k \in H_\lambda$, there exists an $\mathfrak{N} \in S$ such that $\{ p_1,\dots,p_k \} \subset  \mathfrak{N} \prec \mathfrak{H}_\lambda$.\footnote{There are other equivalent ways of defining this kind of stationarity.  See Lemma 0 part (a) of Foreman-Magidor-Shelah~\cite{MR924672}.}  Note that if $|\mathfrak{N}|<\kappa$, then $\mathfrak{H}(\mathfrak{N})$ is both an element and subset of $H_\kappa$; in particular, $b_{\mathfrak{N}}$ is of hereditary cardinality $<\kappa$ for all $b \in \mathfrak{N}$.  

\begin{fact}\label{fact_ElemSubmodule}

Suppose $R$ is a ring of size $<\kappa$, $M$ is an $R$-module, $\mathfrak{N} \in \wp^*_\kappa(H_\lambda)$, and 
\[
\{ R,M \} \subset \mathfrak{N} \prec \mathfrak{H}_\lambda.
\]
Then:
\begin{enumerate}
 \item\label{item_Subset} $R$ is also a subset of $\mathfrak{N}$; in fact any $X \in \mathfrak{N}$ such that $|X|<\kappa$ is also a subset of $\mathfrak{N}$;
 \item $\mathfrak{N} \cap M$ is an elementary submodule of $M$;
 \item  $\sigma_{\mathfrak{N}} \restriction M_{\mathfrak{N}}$ is an elementary embedding (in the language of $R$-modules) from $M_{\mathfrak{N}} \to M$ with image $\mathfrak{N} \cap M$. 
\end{enumerate}
\end{fact}
\begin{proof}
Part \ref{item_Subset} follows from the fact that $\mathfrak{N} \cap \kappa$ is transitive, which was part of the definition of $\wp^*_\kappa(H_\lambda)$; see \cite{Cox_MaxDecon} for details.  It follows that the signature of $M$ is both an element and a subset of $\mathfrak{N}$, which ensures that $\mathfrak{N}$ is closed under Skolem functions for $M$.  The last part is just because for any $b \in \mathfrak{N}$, $\sigma_{\mathfrak{N}}(b_{\mathfrak{N}}) = b \cap \mathfrak{N}$. 
\end{proof}

The following lemma will be used only in the proof of Theorem \ref{thm_SC_Kaplansky}.
\begin{lemma}\label{lem_Absoluteness}
Let $R$ be a ring.
\begin{enumerate}[label=(\roman*)]
 \item\label{item_TransHom} Suppose $H$ is a transitive set and $\pi: A \to B$ is a function such that $\{ R,A,B,\pi \} \subset H$.  Then the statement ``$\pi$ is an $R$-module homomorphism from $A \to B$" is absolute between $(H,\in)$ and the universe of sets.
 
 \item If $\mu$ is an uncountable cardinal such that $R \in H_\mu$, then $\mathfrak{H}_\mu$ correctly computes $\text{Ext}^n_R$ for all $n \in \mathbb{N}$; i.e., for all $R$-modules $X,Y \in \mathfrak{H}_\mu$, $\text{Ext}^n_R(X,Y)$ as computed in $\mathfrak{H}_\mu$ is the real $\text{Ext}^n_R(X,Y)$.

\end{enumerate}

\end{lemma}
\begin{proof}
The statement ``$\pi$ is an $R$-module homomorphism from $A \to B$" is a $\Sigma_0$ statement in the language of set theory, so is absolute between all transitive sets (see Jech~\cite{MR1940513}).  Now suppose $\mu$ is an uncountable cardinal, and $A \in \mathfrak{H}_\mu$ is an $R$-module.  Since $A$ and $R$ both have cardinality $<\mu$, there is a projective resolution $\vec{P}$ of $A$ of hereditary cardinality $<\mu$, and hence $\vec{P} \in  \mathfrak{H}_\mu$.  Moreover, for all $R$-modules $X,Y$ in $\mathfrak{H}_\mu$, the $<\mu$-closure of $\mathfrak{H}_\mu$ ensures that $\text{Hom}_R(X,Y) \subset \mathfrak{H}_\mu$ (and part \ref{item_TransHom} ensures that $\mathfrak{H}_\mu$ and the universe agree about which elements of $\mathfrak{H}_\mu$ count as $R$-module homomorphisms).  This ensures that $\mathfrak{H}_\mu$ correctly computes the relevant homology groups that define $\text{Ext}^n(X,Y)$. 
\end{proof}

\section{Proof of Theorem \ref{thm_SC_Kaplansky}}\label{sec_SC}

Following Viale~\cite{Viale_GuessingModel}, we say that $\mathfrak{N}\in \wp^*_\kappa(H_\lambda)$ is \textbf{0-guessing} if $\mu:=$ (the ordertype of $\mathfrak{N} \cap \lambda$) is a cardinal, and the transitive collapse $\mathfrak{H}(\mathfrak{N})$ of $\mathfrak{N}$ is equal to $\mathfrak{H}_\mu$.  A classic result of Magidor shows that if $\kappa$ is supercompact, then the 0-guessing sets are stationary in $\wp^*_\kappa(H_\lambda)$ for all $\lambda \ge \kappa$ (Magidor's lemma was phrased somewhat differently; see \cite{Cox_MaxDecon} for a quick proof).

\begin{lemma}\label{lem_IntersectExt}
If $R$ is a ring of size $<\kappa$, $A$ and $B$ are $R$-modules, and
\[
\{ R,A,B \} \subset \mathfrak{N} \prec \mathfrak{H}_\lambda
\]
is such that $\mathfrak{N}$ is a 0-guessing model with $\mathfrak{N} \cap \kappa \in \kappa$, then:
\begin{enumerate}[label=(\roman*)]
 \item\label{item_Ext} For all $n \in \mathbb{N}$:
\[
\mathfrak{N} \cap \text{Ext}^n_R(A,B) \ \simeq \ \text{Ext}^n\Big( \mathfrak{N} \cap A, \mathfrak{N} \cap B  \Big).
\]

\noindent In particular, 
\[
\text{Ext}^n_R(A,B) = 0 \ \text{ if and only if } \ \text{Ext}^n_R\big( \mathfrak{N} \cap A, \mathfrak{N} \cap B \big) = 0.
\]

 \item\label{item_Lifts} Every homomorphism from $\mathfrak{N} \cap A \to \mathfrak{N} \cap B$ lifts to a homomorphism from $A \to B$.  In particular, if $B$ is also a \emph{subset} of $\mathfrak{N}$, every homomorphism from $\mathfrak{N} \cap A \to B$ lifts to a homomorphism from $A \to B$.
\end{enumerate}
\end{lemma}
\begin{proof}
Since $|R|<\kappa$, we will without loss of generality assume $R \in H_\kappa$.  Since $R \in \mathfrak{N} \cap H_\kappa$ and $\mathfrak{N} \cap \kappa$ is transitive, it follows that $R \subset \mathfrak{N}$ and $R$ is not moved by the transitive collapsing map of $\mathfrak{N}$.  Also, since $R \subset \mathfrak{N}$, $\mathfrak{N} \cap M$ is an $R$-submodule of $M$ for all $R$-modules $M \in \mathfrak{N}$.

Let $G^n:= \text{Ext}^n_R(A,B)$. As discussed above, $\sigma:=\sigma_{\mathfrak{N}}: \mathfrak{H}(\mathfrak{N}) \to \mathfrak{H}_\lambda$ denotes the inverse of the Mostowski collapsing map.  Since $A$ and $B$ are in $\mathfrak{N}$, it follows that $G^n \in \mathfrak{N}$ for all $n \in \mathbb{N}$, and these objects are all in the range of $\sigma$.  Recall that for $b \in \mathfrak{N}$---i.e., for $b$ in the range of $\sigma$---we use $b_{\mathfrak{N}}$ to denote $\sigma^{-1}(b)$.  Then by Fact \ref{fact_ElemSubmodule},
\begin{equation}\label{eq_3_isos}
R_{\mathfrak{N}}=R, \ A_{\mathfrak{N}} \simeq \mathfrak{N} \cap A, \  B_{\mathfrak{N}} \simeq \mathfrak{N} \cap B, \ \text{ and } G^n_{\mathfrak{N}} \simeq \mathfrak{N} \cap G^n.
\end{equation}

By elementarity of $\sigma$, $\mathfrak{H}(\mathfrak{N}) \models$ ``$G^n_{\mathfrak{N}} = \text{Ext}^n_{R_{\mathfrak{N}}}(A_{\mathfrak{N}},B_{\mathfrak{N}})$".  Since $\mathfrak{N}$ is 0-guessing, $\mathfrak{H}(\mathfrak{N})$ is of the form $\mathfrak{H}_\mu$, so by Lemma \ref{lem_Absoluteness}, $\mathfrak{H}(\mathfrak{N})=\mathfrak{H}_\mu$ is correct about $\text{Ext}^n_{R_{\mathfrak{N}}}$ for all $n \in \mathbb{N}$.  In particular, $G^n_{\mathfrak{N}}$ really is $\text{Ext}^n_{R_{\mathfrak{N}}}(A_{\mathfrak{N}},B_{\mathfrak{N}})$.  This, together with \eqref{eq_3_isos}, proves part \ref{item_Ext}.

For part \ref{item_Lifts}, suppose $\phi: \mathfrak{N} \cap A \to \mathfrak{N} \cap B$ is a homomorphism, and let $\overline{\phi}:= \sigma^{-1}[\phi]$ be the pointwise preimage of $\phi$ via $\sigma$.\footnote{I.e., $\overline{\phi} := \{ \sigma^{-1}(a,b) \ : \ (a,b) \in \phi \}$.}  Since the domain and range of $\phi$ are subsets of $\mathfrak{N}$, it follows that $\overline{\phi}$ is a total function from $A_{\mathfrak{N}} \to B_{\mathfrak{N}}$, and is in fact an $R=R_{\mathfrak{N}}$-module homomorphism.  Since $\overline{\phi}$ is a subset of $A_{\mathfrak{N}} \times B_{\mathfrak{N}}$, $A_{\mathfrak{N}} \times B_{\mathfrak{N}} \in \mathfrak{H}(\mathfrak{N})$, and $\mathfrak{H}(\mathfrak{N})=\mathfrak{H}_\mu$, it follows (by hereditary closure of $\mathfrak{H}_\mu$) that $\overline{\phi}$ is an \emph{element} of $\mathfrak{H}(\mathfrak{N})$.  And, by Lemma \ref{lem_Absoluteness}, ``$\overline{\phi}$ is an $R$-module homomorphism from $A_{\mathfrak{N}} \to B_{\mathfrak{N}}$" is downward absolute from the universe to the transitive set $\mathfrak{H}(\mathfrak{N})$.  Let $\widetilde{\phi}:= \sigma(\overline{\phi})$; by elementarity of $\sigma$, $\mathfrak{H}_\lambda \models$ ``$\widetilde{\phi}$ is an $R$-module homomorphism from $A \to B$", and this is upward absolute to the universe (again by Lemma \ref{lem_Absoluteness}).  And $\widetilde{\phi}$ extends $\phi$ because
\[
\widetilde{\phi} = \sigma\Big( \overline{\phi}  \Big) = \sigma \Big( \sigma^{-1}[\phi] \Big).
\]

\end{proof}

Now to prove Theorem \ref{thm_SC_Kaplansky}:  suppose $R$ is a ring, $\mathcal{S}$ is a set of $R$-modules, and $\kappa$ is a supercompact cardinal such that $R \in H_\kappa$ and $\bigcup \mathcal{S} \in H_\kappa$.  

We claim that ${}^\perp \mathcal{S}$ is a $<\kappa$-Kaplansky class.  Assume $M \in {}^\perp \mathcal{S}$, and $X \subset M$ is such that $|X|<\kappa$.  Fix a cardinal $\lambda \ge \kappa$ such that $M \in H_\lambda$.  By Magidor's lemma mentioned above, there is a 0-guessing $\mathfrak{N} \in \wp_\kappa^*(H_\lambda)$ such that $\mathfrak{N} \prec \mathfrak{H}_\lambda$, $X \in \mathfrak{N}$, and $\mathcal{S} \in \mathfrak{N}$.  By Fact \ref{fact_ElemSubmodule}, $X$ is also a subset of $\mathfrak{N}$.  So it will suffice to prove that $\mathfrak{N} \cap M$ and $\frac{M}{\mathfrak{N} \cap M}$ are both in ${}^\perp \mathcal{S}$.  

Since $\mathcal{S}$ is an element of $\mathfrak{N}$ and has size $<\kappa$, Fact \ref{item_Subset} ensures that $\mathcal{S} \subset \mathfrak{N}$; and then another application of Fact \ref{item_Subset} (this time using that $\bigcup \mathcal{S}$ is in $H_\kappa$) ensures that every $S \in \mathcal{S}$ is both an element \emph{and a subset} of $\mathfrak{N}$.  Consider any such $S$.  By part \ref{item_Ext} of Lemma \ref{lem_IntersectExt},
\begin{equation}
\text{Ext}\big( \mathfrak{N} \cap M, \mathfrak{N} \cap S \boldsymbol{= S} \big) =0, \ \text{ so } \mathfrak{N} \cap M \in {}^\perp S.
\end{equation} 

To show that $\frac{M}{\mathfrak{N} \cap M} \in {}^\perp S$, consider the short exact sequence
\[
\xymatrix{
0 \ar[r] & \mathfrak{N} \cap M \ar[r] & M \ar[r] & \frac{M}{\mathfrak{N} \cap M} \ar[r] & 0
}
\]
and the associated exact sequence
\begin{equation}\label{eq_LongExact}
\begin{gathered}
\xymatrix{
\text{Hom}\big( \mathfrak{N} \cap M, S \big) \ar[drr] & \text{Hom}\big( M,S \big) \ar[l] & \text{Hom}\left(\frac{M}{\mathfrak{N} \cap M}, S \right) \ar[l] \\
\text{Ext}\big( \mathfrak{N} \cap M, S \big) & \text{Ext}\big( M,S \big)=0 \ar[l] & \text{Ext}\left(\frac{M}{\mathfrak{N} \cap M}, S \right). \ar[l] 
}
\end{gathered}
\end{equation}
Observe that the diagonal map is surjective, since $\text{Ext}(M,S)=0$.  Furthermore, since $S$ is both an element and subset of $\mathfrak{N}$, part \ref{item_Lifts} of Lemma \ref{lem_IntersectExt} ensures that the restriction map $\text{Hom}(M,S) \to \text{Hom}(\mathfrak{N} \cap M, S)$ is also surjective.  Then by exactness, the lower right term in the diagram \eqref{eq_LongExact} must be 0.

\section{Proof of Theorem \ref{thm_MainVP_result}}\label{sec_VP}

\subsection{Proof of Theorem \ref{thm_VP_Main}}

$V$ denotes the universe of sets.  We will say $P$ is a \textbf{class relation} if $P$ is a definable subclass of $V^n$ for some (meta-mathematical) natural number $n$, possibly defined with some suppressed parameters.\footnote{I.e., $P$ is a class relation if $P \subseteq V^n$ for some  natural number $n$, and there is a formula $\phi(u_1,\dots,u_n, w_1,\dots,w_k)$ in the language of set theory, and parameters $p_1,\dots,p_k$, such that 
\[
P = \Big\{ (x_1,\dots,x_n) \ : \ \phi(x_1,\dots,x_n, p_1,\dots,p_k)  \Big\}.
\]
}  Recall from Section \ref{sec_Prelims} that for a (partially) elementary submodel $\mathfrak{N}$ of the universe of sets, if $b \in \mathfrak{N}$, then $b_{\mathfrak{N}}$ denotes the image of $b$ under the transitive collapse of $\mathfrak{N}$.  The following lemma is an immediate consequence of Corollary A.2 of Cox~\cite{Cox_MaxDecon}.

\begin{lemma}\label{lem_VP_Reflect}
Assume Vop\v{e}nka's Principle.  Let $P \subseteq V^{n}$ be an $n$-ary class relation.  Then there is a proper class of cardinals $\kappa$ (depending on $P$) with the following property: for every  \[
(a_1,\dots,a_n, r) \in V^n \ \times \ H_\kappa,
\]
there exists an $\mathfrak{N}$ such that $|\mathfrak{N}|<\kappa$, $\mathfrak{N} \cap \kappa \in \kappa$, $\{ a_1,\dots,a_k,r\} \subset \mathfrak{N} \prec_{\Sigma_1} (V,\in)$, and
\[
(a_1,\dots,a_n) \in P \ \iff \ \Big( (a_1)_{\mathfrak{N}},\dots,(a_k)_{\mathfrak{N}}\Big) \in P.
\]

\end{lemma}

Note that, if $\mathfrak{N}$ is as in the conclusion of the lemma, then the transitive collapsing map of $\mathfrak{N}$ fixes $r$, since $r \in \mathfrak{N} \cap H_\kappa$ and $\mathfrak{N} \cap \kappa$ is transitive (Fact \ref{fact_ElemSubmodule}).  In the application below, the role of the $r$ will be played by the ring.  Also note that if $|\mathfrak{N}|<\kappa$, then the transitive collapse of $\mathfrak{N}$ is both an element and subset of $H_\kappa$.

\begin{remark}\label{rem_NonConcrete}
In concrete categories like $R$-Mod, if $M$ is an $R$-module and $R \cup \{ R,M\} \subset \mathfrak{N} \prec_{\Sigma_1} (V,\in)$, then we can talk about the $R$-module $\mathfrak{N} \cap M$, which is isomorphic to $M_{\mathfrak{N}}$.  Similarly, if $f: A \to B$ is an $R$-module homomorphism and $f \in \mathfrak{N}$, then we can make sense of 
\[
f \restriction \mathfrak{N}:  \mathfrak{N} \cap A \to \mathfrak{N} \cap B,
\]
which is isomorphic (in the arrow category) to $f_{\mathfrak{N}}$.

But the transitive collapsed versions make sense even for non-concrete categories, and under VP can be arranged to be legitimate objects and morphisms from that category.  For example, suppose $\{a,f\} \subset \mathfrak{N} \prec_{\Sigma_1} (V,\in)$ where $a$ is an object and $f$ is a morphism in the possibly non-concrete category $\mathcal{C}$.  Then $a_{\mathfrak{N}}$ and $f_{\mathfrak{N}}$ are always defined (simply as the image of $a$, $f$ under the transitive collapse map of $\mathfrak{N}$, respectively).  And, if $\mathfrak{N}$ correctly reflects the definition of the category---as is the case, for example, if the $P$ from Lemma \ref{lem_VP_Reflect} specifies the objects and morphisms of $\mathcal{C}$, and $\mathfrak{N}$ is as in the conclusion of that lemma---then $a_{\mathfrak{N}}$ and $f_{\mathfrak{N}}$ will be in the category $\mathcal{C}$ as well (and have any other properties that $a$ and $f$ had that were specified by $P$).
\end{remark}

We now prove Theorem \ref{thm_VP_Main}.  Assume Vop\v{e}nka's Principle, and suppose $R$ is a ring and $\mathcal{B}$ is a class of $R$-modules such that ${}^\perp B$ is downward closed under elementary submodules for every $B \in \mathcal{B}$.  Let $P$ be the 2-ary class relation
\[
P:= \big\{ (M,B) \ : \ M,B \in  R\text{-Mod}, \  B \in \mathcal{B}, \text{ and } M \notin {}^\perp B  \big\}.
\]
By Lemma \ref{lem_VP_Reflect}, there is a $\kappa$ such that $R \in H_\kappa$ and $\kappa$ has the required properties listed in the lemma with respect to $P$.  Let $\mathcal{B}^{<\kappa}:= \{ B \in \mathcal{B} \ : \ |B|<\kappa \}$.  We claim that 
\begin{equation}\label{eq_MainEqualityB}
{}^\perp \mathcal{B} = {}^\perp \big( \mathcal{B}^{<\kappa} \big),
\end{equation}
which will complete the proof since $\mathcal{B}^{<\kappa}$ has a set of representatives.  The $\subseteq$ direction of \eqref{eq_MainEqualityB} is trivial.  To see the $\supseteq$ direction, suppose $M$ is an $R$-module and $M \notin {}^\perp \mathcal{B}$; then $\text{Ext}(M,B) \ne 0$ for some $B \in \mathcal{B}$.  So 
\[
(M,B) \in P.
\]
By the property of $\kappa$, there is an $\mathfrak{N}$ of size $<\kappa$ such that $\{M,B,R\} \subset \mathfrak{N} \prec_{\Sigma_1} (V,\in)$, $\mathfrak{N} \cap \kappa \in \kappa$, and $\Big( M_{\mathfrak{N}}, B_{\mathfrak{N}} \Big) \in P$; so
\begin{equation}\label{eq_NcapBinB}
B_{\mathfrak{N}} \in \mathcal{B} 
\end{equation}
and
\begin{equation}\label{eq_NonZero}
M_{\mathfrak{N}} \notin {}^\perp \left(B_{\mathfrak{N}}\right).
\end{equation}

\noindent (Viewed another way---assuming without loss of generality that $\mathcal{B}$ is closed under isomorphisms---this just says that $\mathfrak{N} \cap B \in \mathcal{B}$ and $\mathfrak{N} \cap M \notin {}^\perp ( \mathfrak{N} \cap B )$).

We claim that $M \notin {}^\perp \left( B_{\mathfrak{N}} \right)$; this will finish the proof, since $B_{\mathfrak{N}}$ is a $<\kappa$-sized element of $\mathcal{B}$.  Suppose toward a contradiction that $M \in {}^\perp \left( B_{\mathfrak{N}} \right)$.  Since $R = R_{\mathfrak{N}}$ (so in particular $R \subset \mathfrak{N}$), Fact \ref{fact_ElemSubmodule} ensures that $M_{\mathfrak{N}}$ is an $R=R_{\mathfrak{N}}$-module and
\[
\sigma_{\mathfrak{N}} \restriction M_{\mathfrak{N}}: M_{\mathfrak{N}} \to M
\]
is an elementary embedding in the language of $R$-modules, so $M_{\mathfrak{N}}$ is isomorphic to an elementary submodule of $M$.  Since $M \in {}^\perp \left(B_{\mathfrak{N}} \right)$ and ${}^\perp \left(B_{\mathfrak{N}} \right)$ is downward closed under elementary submodules, it follows that $M_{\mathfrak{N}} \in {}^\perp \left(B_{\mathfrak{N}} \right)$.  But this contradicts \eqref{eq_NonZero}.  This completes the proof of Theorem \ref{thm_VP_Main}.

\subsection{Proof of Theorem \ref{thm_MainVP_result}}

Assume VP is consistent.  By the result of Appendix \ref{app_VPDIAMOND}, VP is also consistent with \eqref{eq_DiamondEverywhereIntro}, so we may assume that both VP and \eqref{eq_DiamondEverywhereIntro} hold.  Suppose $(\mathcal{A},\mathcal{B})$ is a cotorsion pair such that ${}^\perp B$ is downward closed under elementary submodules for all $B \in \mathcal{B}$.  By Theorem \ref{thm_VP_Main}, there is a set $\mathcal{B}_0 \subseteq \mathcal{B}$ such that
\[
{}^\perp \mathcal{B} = {}^\perp \mathcal{B}_0.
\]
So the cotorsion pair
\[
(\mathcal{A},\mathcal{B})=\big( {}^\perp \mathcal{B}, \mathcal{B} \big) = \big( {}^\perp \mathcal{B}_0,\mathcal{B} \big)
\]
is cogenerated by the set $\mathcal{B}_0$, and ${}^\perp B$ is downward closed under elementary submodules for all $B \in \mathcal{B}_0$.  It follows that ${}^\perp \mathcal{B}_0$ is downward closed under elementary submodules.  By \eqref{eq_DiamondEverywhereIntro} and the \v{S}aroch-Trlifaj Theorem \ref{thm_SarochTrlifaj}, the cotorsion pair $\big( {}^\perp \mathcal{B}_0,\mathcal{B} \big)$---which is cogenerated by the set $\mathcal{B}_0$---is also generated by a set, and hence complete.

\section{Generalizations and open problems}\label{sec_GenOpen}

Keeping Lemma \ref{lem_VP_Reflect} and Remark \ref{rem_NonConcrete} in mind, Theorem \ref{thm_VP_Main} can be easily generalized to any category where the $\text{Ext}$ functor make sense, so long as for every $b \in \mathcal{B}$, the class ${}^\perp b$ is closed under the $(-)_{\mathfrak{N}}$ operation for sufficiently many $\mathfrak{N} \prec_{\Sigma_1}(V,\in)$.

We end with some questions:
\begin{question}[Trlifaj, personal correspondence]\label{q_Trlifaj}
Is there (a ZFC-provable) example of a ring and a cotorsion pair of modules over that ring that is \textbf{not} generated by a set?\footnote{Or, equivalently, whose left coordinate is not \emph{deconstructible}.} 
\end{question}

\begin{question}[General Salce Problem]\label{q_GeneralSalceProblem}
Is there (a ZFC-provable) example of a ring and a cotorsion pair of modules over that ring that is \textbf{not} complete?\footnote{Equivalently, whose left coordinate is not a special precovering class?} 
\end{question}

By Theorem \ref{thm_MainVP_result}, an affirmative answer to either question would require the ring to \emph{not} be (ZFC-provably) hereditary.  And by Theorem \ref{thm_SC_Kaplansky}, an affirmative answer to either question could \emph{not} be a (ZFC-provably) cotorsion pair of the form 
\[
\Big( {}^\perp \mathcal{S}, \big( {}^\perp \mathcal{S} \big)^\perp \Big)
\]
where $\mathcal{S}$ is a set and ${}^\perp \mathcal{S}$ is closed under direct limits.

We showed that consistency of the large cardinal principle VP implies the consistency of an affirmative solution to Salce's question.  This raises:
\begin{question}
Does the (scheme) ``all cotorsion pairs in the category of abelian groups are complete" carry large cardinal consistency strength?  Does the conclusion of Theorem \ref{thm_VP_Main} carry large cardinal consistency strength?
\end{question}

\appendix

\section{Consistency of Vop\v{e}nka's Principle with Diamond on all stationary sets}\label{app_VPDIAMOND}

For a regular uncountable $\lambda$ and a stationary $S \subseteq \lambda$, Jensen's $\Diamond_\lambda(S)$ principle asserts that there is a sequence
\[
\langle X_\alpha \ : \ \alpha \in S \rangle
\]
such that $X_\alpha \subseteq \alpha$ for all $\alpha \in S$, and for every $X \subseteq \lambda$, the set
\[
\{ \alpha  \in S \ :  \ X \cap \alpha = X_\alpha   \}
\]
is stationary.  Starting with a model of Vop\v{e}nka's Principle (VP), we briefly sketch how to produce a model of VP that also satisfies
\begin{equation}\label{eq_ForceDiamondAtKappa}
\forall \lambda \in \text{REG} \cap [\omega_1,\infty) \ \forall S \subseteq \lambda \ \ S \text{ stationary } \implies \ \Diamond_\lambda(S).
\end{equation}
It is a folklore fact that if $\lambda$ is regular and uncountable, then there is a $<\lambda$-directed closed poset $\mathbb{D}_\lambda$ forcing ``$\Diamond_\lambda(S)$ holds for all stationary $S \subseteq \lambda$"; and if $\lambda^{<\lambda}=\lambda$, then $\mathbb{D}_\lambda$ also has the $\lambda^+$-chain condition ($\mathbb{D}_{<\lambda}$ is just a $<\lambda$-support product of adding Cohen subsets of $\lambda$; in fact, adding a \emph{single} Cohen subset of $\lambda$ suffices, see \cite{125308}).  

Assume VP holds.  By Corollary 26 of Brooke-Taylor~\cite{MR2805294}, we may without loss of generality assume GCH holds as well.  Define an Easton support iteration
\[
\langle \mathbb{P}_\alpha,\dot{Q}_\alpha \ : \ \alpha < \text{ORD} \rangle
\]
where $\mathbb{P}_\alpha$ forces ``if $\alpha$ is regular and uncountable, $\dot{Q}_\alpha = \dot{\mathbb{D}}_\alpha$; otherwise $\dot{Q}_\alpha$ is the trivial forcing".  Let $\mathbb{P}$ be the (class-sized) direct limit of this iteration.  By Theorem 25 of Brooke-Taylor~\cite{MR2805294}, $\mathbb{P}$ preserves VP.  So we just need to show that $\mathbb{P}$ forces \eqref{eq_ForceDiamondAtKappa}.  Suppose $\lambda$ is regular in the extension; then Diamond holds at all stationary subsets of $\lambda$ in $V^{\mathbb{P}_{\lambda+1}}$ by design, since the poset used at stage $\lambda$ was $\mathbb{D}_\lambda^{V_{\mathbb{P}_\lambda}}$.  Note also that $\mathbb{P}_{\lambda+1}=\mathbb{P}_{\lambda^+}$ because $\dot{Q}_\alpha$ is trivial for $\alpha \in (\lambda,\lambda^+)$.  Since ``$\Diamond_\lambda(S)$ holds for all stationary $S \subseteq \lambda$" is clearly preserved by forcings that add no new subsets of $\lambda$, it suffices to show that $\mathbb{P}_{\lambda+1}=\mathbb{P}_{\lambda^+}$ forces the tail of the iteration to be $<\lambda^+$-directed closed.  Now the GCH assumption, regularity of $\lambda$, and the fact that it is an Easton support iteration ensure that $|\mathbb{P}_\lambda| \le \lambda$, so $\mathbb{P}_\lambda$ is (at worst) $\lambda^+$-cc.  Hence, since $\dot{\mathbb{D}}_\lambda$ is also forced to be $\lambda^+$-cc, $\mathbb{P}_{\lambda^+}=\mathbb{P}_{\lambda+1}=\mathbb{P}_\lambda * \dot{\mathbb{D}}_\lambda$ is also $\lambda^+$-cc.  Then the assumptions of Proposition 7.12 of Cummings~\cite{MR2768691} are satisfied (with his $\beta$ and $\kappa$ both interpreted as our $\lambda^+$), which guarantees that $\mathbb{P}_{\lambda^+}$ forces the tail of the iteration to be $<\lambda^+$-directed closed.


\bib{MR1294136}{book}{
  author={Ad\'{a}mek, Ji\v {r}\'{\i }},
  author={Rosick\'{y}, Ji\v {r}\'{\i }},
  title={Locally presentable and accessible categories},
  series={London Mathematical Society Lecture Note Series},
  volume={189},
  publisher={Cambridge University Press, Cambridge},
  date={1994},
  pages={xiv+316},
  isbn={0-521-42261-2},
}

\bib{MR1832549}{article}{
  author={Bican, L.},
  author={El Bashir, R.},
  author={Enochs, E.},
  title={All modules have flat covers},
  journal={Bull. London Math. Soc.},
  volume={33},
  date={2001},
  number={4},
  pages={385--390},
  issn={0024-6093},
}

\bib{MR2805294}{article}{
  author={Brooke-Taylor, Andrew D.},
  title={Indestructibility of Vop\v {e}nka's Principle},
  journal={Arch. Math. Logic},
  volume={50},
  date={2011},
  number={5-6},
  pages={515--529},
  issn={0933-5846},
}

\bib{Cox_MaxDecon}{article}{
  author={Cox, Sean},
  title={Maximum deconstructibility in module categories},
  journal={Journal of Pure and Applied Algebra},
  volume={226},
  year={2022},
}

\bib{MR2768691}{article}{
  author={Cummings, James},
  title={Iterated forcing and elementary embeddings},
  conference={ title={Handbook of set theory. Vols. 1, 2, 3}, },
  book={ publisher={Springer}, place={Dordrecht}, },
  date={2010},
  pages={775--883},
}

\bib{MR2031314}{article}{
  author={Eklof, Paul C.},
  author={Shelah, Saharon},
  title={On the existence of precovers},
  note={Special issue in honor of Reinhold Baer (1902--1979)},
  journal={Illinois J. Math.},
  volume={47},
  date={2003},
  number={1-2},
  pages={173--188},
  issn={0019-2082},
}

\bib{MR1778163}{article}{
  author={Eklof, Paul C.},
  author={Trlifaj, Jan},
  title={Covers induced by Ext},
  journal={J. Algebra},
  volume={231},
  date={2000},
  number={2},
  pages={640--651},
  issn={0021-8693},
}

\bib{MR1798574}{article}{
  author={Eklof, Paul C.},
  author={Trlifaj, Jan},
  title={How to make Ext vanish},
  journal={Bull. London Math. Soc.},
  volume={33},
  date={2001},
  number={1},
  pages={41--51},
  issn={0024-6093},
}

\bib{MR2252654}{article}{
  author={El Bashir, Robert},
  title={Covers and directed colimits},
  journal={Algebr. Represent. Theory},
  volume={9},
  date={2006},
  number={5},
  pages={423--430},
  issn={1386-923X},
}

\bib{MR1926201}{article}{
  author={Enochs, Edgar E.},
  author={L\'{o}pez-Ramos, J. A.},
  title={Kaplansky classes},
  journal={Rend. Sem. Mat. Univ. Padova},
  volume={107},
  date={2002},
  pages={67--79},
  issn={0041-8994},
}

\bib{125308}{misc}{
  title={Forcing Diamond},
  author={Monroe Eskew (https://mathoverflow.net/users/11145/monroe-eskew)},
  note={URL: https://mathoverflow.net/q/125308 (version: 2013-03-22)},
  eprint={https://mathoverflow.net/q/125308},
  organization={MathOverflow},
}

\bib{MR924672}{article}{
  author={Foreman, M.},
  author={Magidor, M.},
  author={Shelah, S.},
  title={Martin's maximum, saturated ideals, and nonregular ultrafilters. I},
  journal={Ann. of Math. (2)},
  volume={127},
  date={1988},
  number={1},
  pages={1--47},
}

\bib{MR2985554}{book}{
  author={G\"{o}bel, R\"{u}diger},
  author={Trlifaj, Jan},
  title={Approximations and endomorphism algebras of modules. Volume 1},
  series={De Gruyter Expositions in Mathematics},
  volume={41},
  edition={Second revised and extended edition},
  publisher={Walter de Gruyter GmbH \& Co. KG, Berlin},
  date={2012},
  pages={xxviii+458},
  isbn={978-3-11-021810-7},
  isbn={978-3-11-021811-4},
}

\bib{MR1940513}{book}{
  author={Jech, Thomas},
  title={Set theory},
  series={Springer Monographs in Mathematics},
  note={The third millennium edition, revised and expanded},
  publisher={Springer-Verlag},
  place={Berlin},
  date={2003},
  pages={xiv+769},
}

\bib{MR565595}{article}{
  author={Salce, Luigi},
  title={Cotorsion theories for abelian groups},
  conference={ title={Symposia Mathematica, Vol. XXIII}, address={Conf. Abelian Groups and their Relationship to the Theory of Modules, INDAM, Rome}, date={1977}, },
  book={ publisher={Academic Press, London-New York}, },
  date={1979},
  pages={11--32},
}

\bib{MR2336972}{article}{
  author={\v {S}aroch, Jan},
  author={Trlifaj, Jan},
  title={Completeness of cotorsion pairs},
  journal={Forum Math.},
  volume={19},
  date={2007},
  number={4},
  pages={749--760},
  issn={0933-7741},
}

\bib{MR3010854}{article}{
  author={\v {S}t'ov\'{\i }\v {c}ek, Jan},
  title={Deconstructibility and the Hill lemma in Grothendieck categories},
  journal={Forum Math.},
  volume={25},
  date={2013},
  number={1},
  pages={193--219},
  issn={0933-7741},
}

\bib{Viale_GuessingModel}{article}{
  author={Viale, Matteo},
  title={Guessing models and generalized {L}aver diamond},
  journal={Ann. Pure Appl. Logic},
  fjournal={Annals of Pure and Applied Logic},
  volume={163},
  year={2012},
  number={11},
  pages={1660--1678},
}



\begin{bibdiv}
\begin{biblist}
\bibselect{../../../MasterBibliography/Bibliography}
\end{biblist}
\end{bibdiv}
\end{document}